\documentclass{amsart}
\usepackage{amssymb,amsmath,graphicx,dsfont}

\usepackage{graphics,url}

\newcounter{mycounter}

\newtheorem{theorem}{Theorem}[section]
\newtheorem{corollary}[theorem]{Corollary}
\newtheorem{lemma}[theorem]{Lemma}

\newtheorem{proposition}[theorem]{Proposition}

\theoremstyle{definition}

\newtheorem{issue}{Issue}

\theoremstyle{remark}
\newtheorem{remark}[theorem]{Remark}

\begin{document}

\title{A core-free semicovering of the Hawaiian Earring}

\author{Hanspeter Fischer}

\address{Department of Mathematical Sciences\\Ball State University\\Muncie, IN 47306\\U.S.A.}

\email{fischer@math.bsu.edu}

\author{Andreas Zastrow}

\address{Institute of Mathematics, University of Gda\'nsk, ul. Wita Stwosza 57,
80-952 Gda\'nsk, Poland}

\email{zastrow@mat.ug.edu.pl}

\thanks{}



\date{\today}

\commby{}

\begin{abstract}
The connected covering spaces of a connected and locally path-connected topological space $X$ can be classified by the conjugacy classes of those subgroups of $\pi_1(X,x)$ which contain an open normal subgroup of $\pi_1(X,x)$, when endowed with the natural quotient topology of the compact-open topology on based loops. There are known examples of semicoverings (in the sense of Brazas) that correspond to open subgroups which do not contain an open normal subgroup. We present an example of a semicovering of the Hawaiian Earring $\mathds{H}$ with corresponding open subgroup of $\pi_1(\mathds{H})$ which does not contain {\em any} nontrivial normal subgroup of $\pi_1(\mathds{H})$.
\end{abstract}

 \maketitle

\section{Introduction}

The fundamental group $\pi_1(X,x)$ of a topological space $X$ with base point $x\in X$ carries a natural topology: considering the space $\Omega(X,x)$ of all continuous loops $\alpha:([0,1],\{0,1\})\rightarrow (X,x)$ in the compact-open topology, we equip $\pi_1(X,x)$ with the quotient topology induced by the function $[\,\cdot\,]:\Omega(X,x)\rightarrow \pi_1(X,x)$ which assigns to each loop $\alpha$ its homotopy class $[\alpha]$. It is known that $\pi_1(X,x)$ need not be a topological group, for example, when $X$ is the Hawaiian Earring \cite{Fabel} (see also \cite{Brazas2011}), although left and right multiplication always constitute homeomorphisms \cite{Calcut}. This problem can be circumvented by removing some of the open subsets from the topology and instead giving $\pi_1(X,x)$ the finest {\em group topology} which makes
$[\,\cdot\,]:\Omega(X,x)\rightarrow \pi_1(X,x)$ continuous \cite{Brazas2013}. However, since both topologies share the same open subgroups \cite[Proposition~3.16]{Brazas2013}, we impose the former.

 For a connected and locally path-connected topological space $X$, the topology of $\pi_1(X,x)$ is intimately tied to the existence of covering spaces: $\pi_1(X,x)$ is discrete if and only if  $X$ is semilocally simply-connected \cite{Calcut}. In turn, $X$ is semilocally simply-connected  if and only if $X$ admits a simply-connected covering space,
  in which case the (classes of equivalent) covering projections $p:\widetilde{X}\rightarrow X$ with connected $\widetilde{X}$ are in one-to-one correspondence with the conjugacy classes of all subgroups of $\pi_1(X,x)$ via the monomorphism on fundamental groups induced by the covering projection \cite[\S2.5]{Spanier}.

 It was stated erroneously in \cite[Theorem~5.5]{Biss} that, in general, the connected covering spaces of a connected and locally path-connected topological space $X$ are in one-to-one correspondence with the conjugacy classes of open subgroups of $\pi_1(X,x)$.
 In fact, it was shown recently that the open subgroups correspond to {\em semicoverings} \cite{Brazas2012} and that they correspond to classical coverings if and only if they contain an open normal subgroup \cite{Torabi} (see also \cite{BrazasNote}). A semicovering $p:\widehat{X}\rightarrow X$ is a local homeomorphism that allows for the unique continuous lifting of paths and their homotopies.

  It was observed in \cite{Torabi} that the solution to \cite[\S1.3 Excercise~6]{Hatcher}, as discussed in \cite[Example~3.8]{Brazas2012}, describes an open subgroup of the fundamental group $\pi_1(\mathds{H})$ of the Hawaiian Earring $\mathds{H}$ which does not contain an open normal subgroup and hence does not correspond to a covering space. In this article, we present a more extreme example: an open subgroup $K$ of $\pi_1(\mathds{H})$ which does not contain {\em any} nontrivial normal subgroup of $\pi_1(\mathds{H})$. We find this subgroup by directly constructing a corresponding  semicovering $q:\widehat{\mathds{H}}\rightarrow \mathds{H}$.

In the last section, we briefly sketch a unified proof of both abovementioned correspondence results of \cite{Brazas2012} and \cite{Torabi} for connected and locally path-connected spaces (Corollaries~\ref{B} and \ref{T} below, respectively) from the common perspective of the further generalized covering spaces of \cite{FischerZastrow}. We thereby hope to bring out the subtle difference in the classical construction (as discussed in \cite{Spanier}) of  (semi)covering spaces corresponding to open versus  open {\em normal} subgroups of the fundamental group.

\section{The graph $\widehat{\mathds{H}}$}\label{graph}
Consider the Hawaiian Earring, i.e., the planar space $\mathds{H}=\bigcup_{i\in \mathbb{N}} C_i\subseteq \mathbb{R}^2$ where  $C_i= \{(x,y)\in \mathbb{R}^2\mid x^2+(y-1/i)^2=1/i^2\}$ with base point ${\bf 0}=(0,0)$. (See Figure~\ref{H}.) For each $i\in \mathbb{N}$, consider the parametrization $l_i:[0,1]\rightarrow C_i$ defined by $l_i(t)=((\sin 2\pi t)/i, (1-\cos 2\pi t)/i)$.
\vspace{-15pt}

\begin{figure}[h]\includegraphics{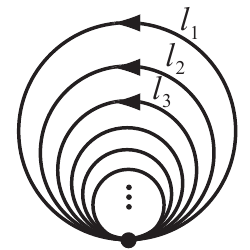}
\vspace{-10pt}

\caption{\label{H} The Hawaiian Earring $\mathds{H}$}
\end{figure}

\vspace{-20pt}

\begin{figure}[h]\includegraphics{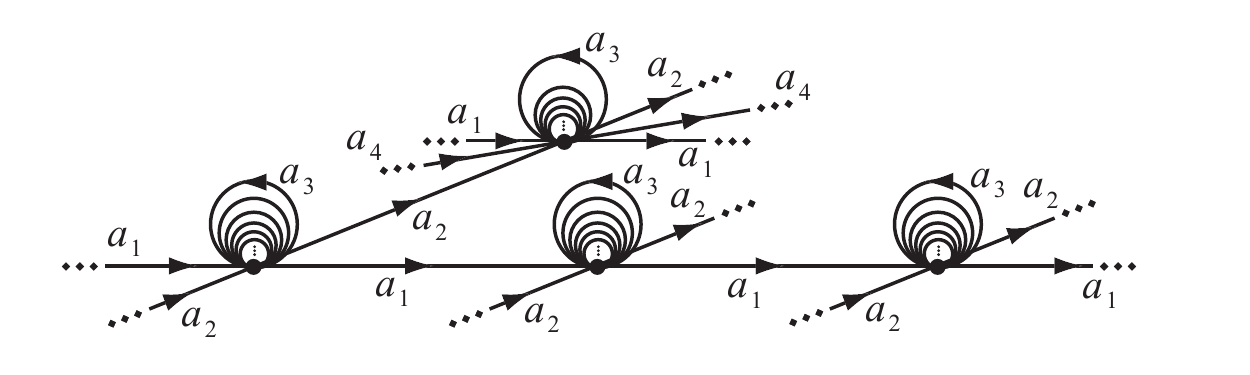}
\vspace{-25pt}

\caption{\label{Hhat} A detail of the graph $\widehat{\mathds{H}}$ with every $a_i$ mapping to $l_i$}
\end{figure}

Before we begin with the construction of the graph $\widehat{\mathds{H}}$, we outline some guiding principles. Since $\widehat{\mathds{H}}$ is to be locally homeomorphic to $\mathds{H}$, all but finitely many edges of $\widehat{\mathds{H}}$ which are incident to any given vertex must be looping edges---we choose to assemble all non-looping edges into a tree. (See Figure~\ref{Hhat}.) This tree, denoted by $\Gamma^\ast$ below, will be constructed in Steps~1--4 as a subtree of the Cayley graph $\Gamma$ for the free group on $\{a_1, a_2, a_3, \cdots\}$. (Step 5 completes the construction of $\widehat{\mathds{H}}$ by attaching the looping edges to $\Gamma^\ast$.) There are two competing demands on the structure of $\Gamma^\ast$.\linebreak On one hand, to secure the desired path lifting property, the branching of $\Gamma^\ast$ must be limited, so as to keep paths in $\widehat{\mathds{H}}$ from running off to infinity along non-looping edges whose images in $\mathds{H}$ form a sequence of circles whose diameters converge to zero. (For example, one would not be able to lift the continuous loop $\ell=(l_1\cdot l_2\cdot l_3\cdots)\cdot(l_1\cdot l_2\cdot l_3\cdots)$ in $\mathds{H}$ to a continuous path $\widehat{\ell}:([0,1],0)\rightarrow (\widehat{\mathds{H}},v)$ if the edge-path $a_1a_2a_3\cdots $ in $\widehat{\mathds{H}}$, starting at vertex $v$, were to contain infinitely many non-looping edges.) On the other hand, for $K$ not to contain any nontrivial normal subgroup of $\pi_1(\mathds{H})$, every essential loop in $\mathds{H}$ must have at least one lift in $\widehat{\mathds{H}}$ which is not a loop. (If there were an essential loop $\ell$ in $\mathds{H}$ with only loops as lifts, then for any path $\alpha$ in $\mathds{H}$ and its reverse $\alpha^-$, all lifts of $\alpha\cdot\ell\cdot \alpha^-$ and of $\alpha\cdot\ell^-\cdot \alpha^-$ would be loops,
 making the normal subgroup generated by the homotopy class of $\ell$, which we may assume to be based at the origin, a nontrivial normal subgroup of $K$.) In fact, as we shall see in the proof of Proposition~\ref{normal} below, it suffices to incorporate into $\Gamma^\ast$ one lift $\widehat{\ell}$ of each essential finite edge-loop $\ell$ in $\mathds{H}$, provided we simultaneously arrange for all (except the largest two) circles of $\mathds{H}$ which are smaller than the smallest circle crossed by $\ell$, to lift to loops at all vertices along $\widehat{\ell}$. Formally, we proceed as follows.

Let $F_\infty$ be the free group on the countably infinite set ${\mathcal A}=\{a_1, a_2, a_3, \dots\}$. Let $\mathcal W_n$ be the set of all finite words over the finite alphabet $\{a_1^{\pm 1}, a_2^{\pm 1}, \dots, a_n^{\pm 1}\}$ and let $\mathcal W=\bigcup_{n=1}^\infty \mathcal W_n$. For $w\in \mathcal W$,  let $w'$ denote the word resulting from completely reducing $w$, using the usual cancellation operations.
 Then the vertex set $V$ of the directed Cayley graph $\Gamma$ for the group $F_\infty$, with respect to the generating set $\mathcal A$, consists of all words $w$ in $\mathcal W$ which are reduced (i.e., $w=w'$) and its directed edge set is given by $E=\{(u,v)\mid u, v\in V, v=(ua_i)'$  for some $i\in \mathbb{N}\}$.  We label the directed edge $(u,v)\in E$  from $u$ to $v=(ua_i)'$ by $a_i$. Note that the underlying undirected graph for $\Gamma$ is a tree all of whose vertices have valence $\aleph_0$. We denote the empty word by $\mathds{1}$ and the length of a word $w\in \mathcal W$ by $|w|$.

Let $\{w_1, w_2, w_3, \dots\}$ be a complete list of all non-empty words in $\mathcal W$. For each $j\in \mathbb{N}$, let \[\widehat{w}_j=a_1a_2a_1a_2a_1\cdots\] be the finite word of length \[|\widehat{w}_j|=2(|w_1|+|w_2|+\cdots + |w_{j-1}|)+3j+|w_j|\] whose letters alternate between $a_1$ and $a_2$. Then $\widehat{w}_j\in V$.

We define the graph $\widehat{\mathds{H}}$ based on $\Gamma$ in five steps:\vspace{5pt}

{\bf Step 1:} Let $Z_j$ be the set of vertices visited by the edge-path in $\Gamma$ which starts at vertex $\widehat{w}_j$ and follows the edges $x_1, x_2, \dots, x_m\in \{a_1^{\pm 1},a_2^{\pm 1}, a_3^{\pm 1},\dots \}$ which appear as the letters of the word $w_j=x_1x_2\cdots x_m$. (Here,  $a_i^{-1}$ means $a_i$ in reverse.) That is,\linebreak\vspace{-20pt}

 \[Z_j=\{u\in V\mid u=(\widehat{w}_jx_1x_2\cdots x_i)' \mbox{ for some } 0\leqslant i \leqslant m\}.\] (Note that the same vertex $u\in Z_j$ might be visited multiple times by this edge-path, for different values of $i$, since $w_j$ need not be a reduced word.)

Consider the set $Z$ of vertices visited by the infinite zig-zag ray in $\Gamma$, which starts at vertex $\mathds{1}$ and follows the alternating edges $a_1, a_2, a_1, a_2, a_1,\dots$. Then $Z$ intersects each $Z_j$ in at least one vertex. For $i<j$, we have $|\widehat{w}_j|-|\widehat{w}_i|\geqslant |w_j|+|w_i|+3$, so that $Z_i$ and $Z_j$ are separated in the tree $\Gamma$ by at least three consecutive edges whose vertices are in $Z$.
Hence, the sets $Z_1, Z_2, Z_3, \dots$ are pairwise disjoint.\vspace{5pt}

{\bf Step 2:} We add vertices to each $Z_j$ along finitely many ``straight lines'' through the vertices of $Z_j$. Specifically, given $w_j=x_1x_2\cdots x_m$, choose $n\geqslant 2$ minimal with $x_i\in \{a_1^{\pm 1}, a_2^{\pm 1},\dots, a_n^{\pm 1}\}$ for all $1\leqslant i\leqslant m$. For each $u\in Z_j$ and each $1\leqslant s \leqslant n$, let $L_{u,s}$ be the set of vertices visited by the two infinite rays of $\Gamma$ that start at vertex $u$ and follow the edges $a_s, a_s, a_s,\dots$ and $a_s^{-1}, a_s^{-1}, a_s^{-1}, \dots$, respectively. That is, \[L_{u,s}=\{v\in V\mid v=(ua_s^r)' \mbox{ and } r\in \mathbb{Z}\}.\] (Note that $L_{u,s}$ might meet $Z_j$ in more than one vertex. Indeed, if $u_1,u_2\in Z_j$ are connected by an edge labeled $a_s$, then $L_{u_1,s}=L_{u_2,s}$.) We define \[Y_j=Z_j\;\cup \bigcup_{u\in Z_j, 1\leqslant s\leqslant n} L_{u,s}.\]

If $Z$ meets any given $L_{u,s}$, then it does so in at most two (adjacent) vertices, at least one of which belongs to the corresponding $Z_j$ with $u\in Z_j$. This fact, when combined with the estimate at the end of Step 1, implies that for $i\neq j$, the sets $Y_i$ and $Y_j$ are separated in the tree $\Gamma$ by some edge with vertices in $Z$. In particular, the sets $Y_1, Y_2, Y_3, \dots$ are pairwise disjoint.
 We define \[Y=Z\cup \bigcup_{j=1}^\infty Y_j.\]

{\bf Step 3:} Starting with $\Gamma_0=\Gamma$, we inductively define a sequence $\Gamma_0,\Gamma_1, \Gamma_2, \Gamma_3, \dots$ of successive subtrees, each of which contains all vertices in $Y$. For $j\in \mathbb{N}$, we let $\Gamma_j$ be the graph obtained from $\Gamma_{j-1}$ by removing all vertices $v$ which are at edge-distance 1 from $Y_j$, unless $v$ is connected to $Y_j$ by an edge labeled $a_1$ or $a_2$ (having $v$ either as terminal or initial vertex), along with all edges incident to $v$ and all vertices and edges of $\Gamma_{j-1}$ which are separated from $Y_j$ in $\Gamma_{j-1}$ by $v$.

Put differently (with $n$ as in Step~2 for $Y_j$), we let $\Gamma_j$ be the graph obtained from $\Gamma_{j-1}$ by removing all vertices $v$ (along with all incident edges) of the (reduced) form \begin{equation}\label{form} v=(v_1v_2\cdots v_t)(v_{t+1}v_{t+2}\cdots v_{t+r})v_{t+r+1}\cdots\end{equation}
\hspace{.7in} \parbox{4.5in}{\hspace{2pt} for which $v_1v_2\cdots v_t\in Z_j$, and $v_1v_2\cdots v_{t+1}\not\in Z_j$, and

\begin{itemize}
\item[either] $\displaystyle \left\{\begin{array}{l}
                 r=0 \mbox{ and} \\
                 v_{t+1}=v_{t+r+1}=a_k^{\pm 1} \mbox{ for some } k>n
               \end{array}\right.$\vspace{5pt}

\item[or \hspace{5pt} ] $\displaystyle \left\{\begin{array}{l}
                r>0 \mbox{ and}\\
                v_{t+1}v_{t+2}\cdots v_{t+r}=a_s^{\pm r} \mbox{ for some } 1\leqslant s\leqslant n \mbox{ and} \\
                 v_{t+r+1}=a_k^{\pm 1} \mbox{ for some } k\notin\{1,2,s\}
                 \end{array}\right.$
\end{itemize}}

\begin{remark}\label{starremark} \hspace{5pt} \begin{itemize}
\item[(i)] The expression $(v_1v_2\cdots v_t)(v_{t+1}v_{t+2}\cdots v_{t+r})v_{t+r+1}\cdots$ in Formula (\ref{form}) is  a finite word of length at least $t+r+1$, which begins according to the specified conditions and continues in any way that forms an overall reduced word.
\item[(ii)]
 We briefly verify the correctness of Formula (\ref{form}): For $u\in Z_j$ and $1\leqslant s\leqslant n$, $L_{u,s}\cap Z_j=\{u_1, u_2,\dots,u_\ell\}$ is a finite set (containing $u$) of consecutively adjacent vertices (in the tree $\Gamma$), separating $L_{u,s}\setminus Z_j$ into sets $L_{u,s}^+$ and $L_{u,s}^-$ (distinguished by direction) with $u_1$ and $u_\ell$ adjacent to $L_{u,s}^-$ and $L_{u,s}^+$, respectively. Each $u_i$ appears in  (\ref{form}) as some $v_1v_2\cdots v_t$ with $r=0$, and each vertex in $L_{u,s}^+$ (resp. $L_{u,s}^-$) appears as some $v_1v_2\cdots v_{t+r}$ with $r>0$, where $v_1v_2\cdots v_t=u_\ell$ (resp. $=u_1$) and $v_{t+1}v_{t+2}\cdots v_{t+r}=a_s^{+r}$  (resp.  $=a_s^{-r}$).\linebreak As for the corresponding range of $k$ in the equation $v_{t+r+1}=a_k^{\pm 1}$, note that $u_i\in L_{u_i,k}$ for $1\leqslant k \leqslant n$ while $L_{u,s}^{\pm}\cap L_{\mu,k}=\emptyset$ for $\mu\in Z_j$ and $k\neq s$.
\item[(iii)] This procedure never removes a vertex $v\in Y_i$ with $i\neq j$, for if $Y_i$ is connected to $Y_j$ by an edge of $\Gamma$, then this edge is labeled $a_1$ or $a_2$ by Step 2.\linebreak \vspace{-10pt}

    \end{itemize} Hence, $\Gamma_j$ is a (connected) subtree of $\Gamma_{j-1}$ containing all vertices in $Y$.

\end{remark}

  {\bf Step 4:} At this point,  a vertex $v$ of the tree $\bigcap_{j=1}^\infty \Gamma_j$ has finite valence if and only if $v\in \bigcup_{j=1}^\infty Y_j$;  all other vertices are still incident with the same edges as in $\Gamma$.\linebreak We now prune back to a further subtree, $\Gamma^\ast$, which still contains the vertices in $Y$, leaves the valence of every vertex  $v\in \bigcup_{j=1}^\infty Y_j$ unchanged, but is such that all other vertices have valence 4 and are incident only with edges labeled $a_1$ or $a_2$.

   Specifically, we let $\Gamma^\ast$ be the graph obtained from $\bigcap_{j=1}^\infty \Gamma_j$ by removing all vertices $v$ (along with all incident edges) of the (reduced) form $v=ua_k^{\pm 1}w$, where $u$ represents a vertex of $\bigcap_{j=1}^\infty \Gamma_j$ of infinite valence, $k\not\in \{1,2\}$ and $w\in \mathcal W$.

Due to the symmetry of our construction, for every vertex $v$ of $\Gamma^\ast$, there is a finite subset $E_v\subseteq \mathbb{N}$ such that
the directed edges of $\Gamma^\ast$ terminating in $v$ are labeled by the elements of the set $\{a_i\mid i\in E_v\}$ and the directed edges of $\Gamma^\ast$ emanating from $v$ are also labeled by the elements of the set $\{a_i\mid i\in E_v\}$.\vspace{5pt}

{\bf Step 5:} For each vertex $v$ of $\Gamma^\ast$ and each $i\in\mathbb{N}\setminus E_v$, we add one additional directed edge $(v,v)$ to $\Gamma^\ast$ which loops from $v$ back to $v$, and label it by $a_i$. The geometric realization of the resulting graph will be denoted by $\widehat{\mathds{H}}$ and will be given a natural non-CW (metrizable) topology in the next section. We choose the vertex $\mathds{1}\in \widehat{\mathds{H}}$ as the base point.\vspace{5pt}

Edge-paths in $\widehat{\mathds{H}}$ are understood to be edge-paths in the underlying undirected graph, i.e., directed edges may be traversed in both directions. Upon specifying a starting vertex, edge-paths are represented by elements of ${\mathcal W}$, whose letters indicate which edges are traversed in what order and direction. We record for later reference:

\begin{lemma}\label{islands} Let $\widehat{Y}_j$ be the (connected) subgraph of $\widehat{\mathds{H}}$ spanned by the vertices in $Y_j$.\linebreak
Then the collection $\{\widehat{Y}_j\mid j\in \mathbb{N}\}$ is pairwise disjoint and has the following properties:
\begin{list}{\em (\arabic{mycounter})}{\usecounter{mycounter}}\setlength{\itemindent}{-15pt}
\item\label{base} $\mathds{1}\not\in \widehat{Y}_j$ for all $j$.
\item\label{large} If $v$ is a vertex of $\widehat{\mathds{H}}$ with $v\not\in \widehat{Y}_j$ for all $j$, then $E_v=\{1,2\}$.
\item\label{small} If $v$ is a vertex of $\widehat{Y}_j$ and $w_j\in {\mathcal W}_n$ $(n\geqslant 2)$, then $\{1,2\}\subseteq E_v\subseteq \{1,2,\dots,n\}$.
\item \label{separate} Every edge-path in $\widehat{\mathds{H}}$ from $\widehat{Y}_j$ to $\widehat{\mathds{H}}\setminus \widehat{Y}_j$ contains a label $a_1^{\pm 1}$ or $a_2^{\pm 1}$.
\item \label{core} The graph $\widehat{Y}_j$ contains the edge-path which starts at vertex $\widehat{w}_j$ and follows the letters of the word $w_j$; this edge-path lies entirely in the subtree $\Gamma^\ast$.
\end{list}
\end{lemma}

\begin{proof}
First observe that $Y_j$ is the vertex set of $\widehat{Y}_j$. Items (\ref{base}) and (\ref{core}) are clear.

 To prove (\ref{large}) and (\ref{small}), let  $w_j=x_1x_2\cdots x_m$ and choose $n\geqslant 2$ minimal with $x_i\in \{a_1^{\pm 1}, a_2^{\pm 1},\dots, a_n^{\pm 1}\}$ for all $1\leqslant i\leqslant m$. By Steps 2 and 3, every vertex $u\in Z_j$ has $E_u=\{1,2,\cdots, n\}$ and every vertex $v\in L_{u,s}$ with $u\in Z_j$, $1\leqslant s\leqslant n$ and $v\not\in Z_j$ has $E_v=\{1,2,s\}$. By Step 4, every vertex $v\not\in \bigcup_{i=1}^\infty Y_i$  has $E_v=\{1,2\}$.

 Item (\ref{separate}) follows from (\ref{large}) and Remark~\ref{starremark}(iii).
\end{proof}

\section{The map $q:\widehat{\mathds{H}}\rightarrow \mathds{H}$}\label{map}

 We define the following subsets of $C_i\subseteq \mathds{H}$:
\begin{eqnarray*}
V_i^-&=&l_i([0,3/8))\\
V_i^+&=&l_i((5/8,1])\\U_i&=&l_i((1/4,3/4))
\end{eqnarray*}
We also define the following subsets of $\mathds{H}$:
\[U_{n+1}^\infty=\bigcup_{i=1}^n (V_i^-\cup V_i^+)\cup \bigcup_{i=n+1}^\infty C_i\]
 Then, for every $n\in \mathbb{N}$, $\{U_1, U_2, \dots, U_n,U_{n+1}^\infty\}$ is an open cover of $\mathds{H}$. Note that $U_{m+1}^\infty\supseteq U_{n+1}^\infty$ if $m\leqslant n$ and that  $U_i\cap U_j=\emptyset$ for $i\neq j$.

 Let $e$ be a directed edge (or loop) of $\widehat{\mathds{H}}$, labeled $a_i$, with corresponding parametrization $\psi_e:[0,1]\rightarrow e\subseteq \widehat{\mathds{H}}$. We define the following subsets of $e$:
 \begin{eqnarray*}
 V_e^-&=&\psi_e([0,3/8))\\
 V_e^+&=&\psi_e((5/8,1])\\
  U_e&=&\psi_e((1/4,3/4))
    \end{eqnarray*}
    Accordingly, we obtain bijections $q_e:U_e\rightarrow U_i$, $q_e^-:V_e^-\rightarrow V_i^-$ and $q_e^+:V_e^+\rightarrow V_i^+$ by composing $l_i$ with the inverse of the respective restriction of $\psi_e$.

For each vertex $v$ of $\widehat{\mathds{H}}$, let $U_v$ be the union of all $V_e^-$, where $e$ ranges over all edges of $\widehat{\mathds{H}}$ (including loops) that emanate from $v$,  together with all $V_e^+$, where $e$ ranges over all edges of $\widehat{\mathds{H}}$ (including loops) that terminate in $v$,
together with all entire loops $e$ at $v$ that are labeled $a_i^{\pm 1}$ with $i>n$, where $n\geqslant 2$ is chosen minimal with $E_v\subseteq \{1,2,\dots,n\}$. Let $q_v:U_v\rightarrow U_{n+1}^\infty$ be the unique bijection which agrees with all $q_e$, $q_e^-$ and $q_e^+$, where defined. Note that $q_v(v)={\bf 0}$. (See Figure~\ref{U}.)

\begin{figure}[h]\includegraphics{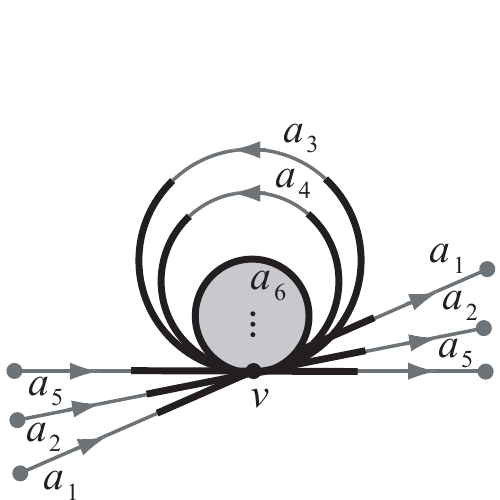} \includegraphics{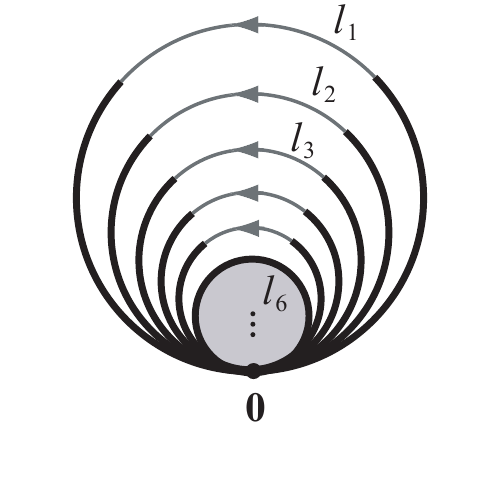}
\vspace{-15pt}

\caption{\label{U} $q_v:U_v\rightarrow U_6^\infty$ with $E_v=\{1,2,5\}$}
\end{figure}

 The collection $\{U_e\mid e$ an edge of $\widehat{\mathds{H}}\}$ is pairwise disjoint, and the same is true for $\{U_v\mid v$ a vertex of $\widehat{\mathds{H}}\}$; together these two collections cover $\widehat{\mathds{H}}$. Moreover, $q_e|_{U_e\cap U_v}=q_v|_{U_e\cap U_v}$ for all $e$ and $v$. Hence, we may
 define a function $q:\widehat{\mathds{H}}\rightarrow \mathds{H}$ by $q(\widehat{x})=q_e(\widehat{x})$ if $\widehat{x}\in U_e$ for some $e$ and $q(\widehat{x})=q_v(\widehat{x})$ if $\widehat{x}\in U_v$ for some $v$.

 We endow $\widehat{\mathds{H}}$ with the unique topology that makes every $U_e$ and every $U_v$ open and which makes every $q|_{U_e}=q_e:U_e\rightarrow U_i$ and every $q|_{U_v}=q_v:U_v\rightarrow U_{n+1}^\infty$ a homeomorphism. In particular, we have the following:

 \begin{proposition}\label{local}
 $q:\widehat{\mathds{H}}\rightarrow \mathds{H}$ is a local homeomorphism.
 \end{proposition}

\begin{proposition}\label{path}
For every continuous path $f:([0,1],0)\rightarrow (\mathds{H},{\bf 0})$ there is a unique continuous lift $\widehat{f}:([0,1],0)\rightarrow (\widehat{\mathds{H}},\mathds{1})$ such that $q\circ \widehat{f}=f$.
\end{proposition}

\begin{proof} We only need to show the existence of $\widehat{f}$, since uniqueness follows from Proposition~\ref{local} and the fact that $\widehat{\mathds{H}}$ is Hausdorff. Choose a partition $0=t_0< t_1< t_2 < \cdots < t_m=1$ of $[0,1]$ such that each $f([t_i,t_{i+1}])$ lies in one of $U_1$, $U_2$ or $U_3^\infty$. Combining  subintervals, if necessary, we may assume that $f([t_{2k},t_{2k+1}])\subseteq U_3^\infty$ and $f([t_{2k+1},t_{2k+2}])\subseteq U_{i_k}$ for all $k$.

By Lemma~\ref{islands}, parts (\ref{base}) and (\ref{large}), $E_\mathds{1}=\{1,2\}$ so that $q_\mathds{1}:U_\mathds{1}\rightarrow U_3^\infty$ is a homeomorphism. Since $f([0,t_1])\subseteq U_3^\infty$, we may define $\widehat{f}|_{[0,t_1]}=q_{\mathds{1}}^{-1}\circ f|_{[0,t_1]}$.

Since $f(t_1)\in U_1\cup U_2$, there is a unique edge $e$ of $\widehat{\mathds{H}}$ with $\widehat{f}(t_1)\in U_e$. Then $f([t_1,t_2])\subseteq q_e(U_e)$, so that we may define $\widehat{f}|_{[t_1,t_2]}=q_{e}^{-1}\circ f|_{[t_1,t_2]}$.
Since $f(t_2)\in (U_1\cup U_2)\cap U_3^\infty$, we have $\widehat{f}(t_2)\in V_e^-\cup V_e^+$. Hence, there is a unique vertex $v$ of $\widehat{\mathds{H}}$ such that $\widehat{f}(t_2)\in U_v$. We now define $\widehat{f}$ on $[t_2,t_3]$.

If $E_v=\{1,2\}$,  then $q_v:U_v\rightarrow U_3^\infty$ is a homeomorphism and we may define $\widehat{f}|_{[t_2,t_3]}=q_{v}^{-1}\circ f|_{[t_2,t_3]}$, since $f([t_2,t_3])\subseteq U_3^\infty$.
Otherwise, by Lemma~\ref{islands}(\ref{large}), we have $v\in \widehat{Y}_i$ for some $i$. In this case,  we choose $n$ such that $w_i\in {\mathcal W}_n$ ($n\geqslant 2$). Then $E_u\subseteq \{1,2,\dots,n\}$ for all vertices $u\in \widehat{Y}_i$, by Lemma~\ref{islands}(\ref{small}), so that $ U_{n+1}^\infty\subseteq q_{u}(U_u)$ for all vertices $u\in \widehat{Y}_i$, due to the minimality condition in the definition of $q_u$. Accordingly, we  choose a partition $t_2=t_2^0<t_2^1<t_2^2<\cdots < t_2^r=t_3$ of $[t_2,t_3]$ such that each  $f([t_2^i,t_2^{i+1}])$ lies in one of $U_3, U_4, \cdots, U_n, U_{n+1}^\infty$. Combining subintervals, if necessary, we may assume that $f([t_2^{2s},t_2^{2s+1}])\subseteq U_{n+1}^\infty$ and $f([t_2^{2s+1},t_2^{2s+2}])\subseteq U_{j_s}$ for all $s$.  Lemma~\ref{islands}(\ref{separate}) allows us now to iteratively define $\widehat{f}|_{[t_2, t_3]}$ by $\widehat{f}|_{[t_2^{2s},t_2^{2s+1}]}=q_{v_s}^{-1}\circ f|_{[t_2^{2s},t_2^{2s+1}]}$ and $\widehat{f}|_{[t_2^{2s+1},t_2^{2s+2}]}=q_{e_s}^{-1}\circ f|_{[t_2^{2s+1},t_2^{2s+2}]}$ with edges $e_s$ that form an edge-path in $\widehat{\mathds{H}}$ through vertices $v_s\in \widehat{Y}_i$.

Processing the remaining subintervals of the partition $0=t_0< t_1 < \cdots < t_m=0$ in the same way, i.e., possibly once further subdividing each $[t_{2s},t_{2s+1}]$, we arrive at the desired lift $\widehat{f}$.
\end{proof}

\begin{proposition}\label{homotopy}
For every continuous homotopy $F:([0,1]^2,(0,0))\rightarrow (\mathds{H},{\bf 0})$ there is a unique continuous lift $\widehat{F}:([0,1]^2,(0,0))\rightarrow (\widehat{\mathds{H}},\mathds{1})$ such that $q\circ \widehat{F}=F$.
\end{proposition}

\begin{proof}
The proof is essentially the same as that of Proposition~\ref{path}. We begin with a partition $0=t_0< t_1< t_2 < \cdots < t_m=1$ of $[0,1]$ such that for each subdivision rectangle $A=[t_i,t_{i+1}]\times [t_j,t_{j+1}]\subseteq [0,1]^2$, there is an element $U_A\in \{U_1, U_2, U_3^\infty\}$ with $F(A)\subseteq U_A$. Subdividing further, if necessary, we may assume that if $A$ and $B$ are two subdivision rectangles of $[0,1]^2$ such that $A\cap B=\{z\}$ is a singleton and such that $U_A=U_B$, then there is a third subdivision rectangle $C$ of $[0,1]^2$ with $A\neq C\neq B$, $A\cap C\cap B=\{z\}$ and $U_A=U_C=U_B$. This allows us to combine the subdivision rectangles into components of constant $U_A$-value, separated by pairwise disjoint  edge-paths in the subdivision grid of $[0,1]^2$.
Analogous to the proof of Proposition~\ref{path}, we start with the component that contains $(0,0)$ and lift one neighboring component at a time, possibly once further subdividing the rectangles of those components with $U_A$-value equal to $U_3^\infty$. While these components might be nested,
the iterative lifting process can be carried out consistently, because each new neighboring component of constant $U_A$-value meets the already lifted region $R$ in exactly one complete  component of the topological boundary of $R$ in $[0,1]^2$. (Note that no two components with distinct constant $U_A$-value from the set $\{U_1, U_2, U_3, U_4,\cdots\}$ are adjacent.)
\end{proof}

\begin{remark}
It is evident from Proposition~\ref{local} and the proofs of Propositions~\ref{path} and \ref{homotopy} that  $q:\widehat{\mathds{H}}\rightarrow \mathds{H}$ is a semicovering in the sense of \cite{Brazas2012}. (See also Remark~\ref{Hausdorff}.)
\end{remark}

\section{The core-free open subgroup $K$ of $\pi_1(\mathds{H},{\bf 0})$}\label{subgroup}

 By Propositions~\ref{path} and \ref{homotopy}, we may define \[K=\{[\alpha]\in \pi_1(\mathds{H},{\bf 0})\mid \widehat{\alpha}(1)=\mathds{1}\}=q_{\#}(\pi_1(\widehat{\mathds{H}},\mathds{1})).\]

The following proposition follows from the fact that $q:\widehat{\mathds{H}}\rightarrow \mathds{H}$ is a semicovering (cf. Initial Step of \cite[Theorem~5.5]{Brazas2012}). For completeness and for later reference, we include a direct proof.

\pagebreak

\begin{proposition}\label{open}
  $K$ is open in $\pi_1(\mathds{H},{\bf 0})$.
\end{proposition}

\begin{proof}
Let $h:\Omega(\mathds{H},{\bf 0})\rightarrow \pi_1(\mathds{H},{\bf 0})$ denote the quotient map. We wish to show that $h^{-1}(K)$ is open in $\Omega(\mathds{H},{\bf 0})$. To this end, let $\alpha\in h^{-1}(K)$. Then $[\alpha]=h(\alpha)\in K$, so that $\widehat{\alpha}(1)=\mathds{1}$. By the proof of Proposition~\ref{path}, there is a partition $0=t_0<t_1<t_2<\cdots < t_m=1$ of $[0,1]$, and edges $e_i$ in $\widehat{\mathds{H}}$ forming an edge-path through vertices $v_i$ such that $\alpha([t_{2k},t_{2k+1}])\subseteq q_{v_k}(U_{v_k})$ and $\alpha([t_{2k+1},t_{2k+2}])\subseteq q_{e_k}(U_{e_k})$ for $0\leqslant k \leqslant (m-1)/2$. Since $\widehat{\alpha}(1)=\mathds{1}$, we have $v_m=\mathds{1}$. Since  $q_{v_k}(U_{v_k})$ and $q_{e_k}(U_{e_k})$ are open in $\mathds{H}$, we may define an open subset $W$ of $\Omega(\mathds{H},{\bf 0})$ by \[W=\bigcap_{k=0}^{(m-1)/2} S([t_{2k},t_{2k+1}],q_{v_k}(U_{v_k}))\cap S([t_{2k+1},t_{2k+2}], q_{e_k}(U_{e_k})),\]
where $S(A,B)=\{f\in \Omega(\mathds{H},{\bf 0})\mid f(A)\subseteq B\}$. Then every $\beta\in W$ can be lifted to $\widehat{\beta}:([0,1],0)\rightarrow (\widehat{\mathds{H}},\mathds{1})$ with $q\circ\widehat{\beta}=\beta$ on the same subdivision intervals and through the same sequence of homeomorphisms as $\alpha$, so that $\widehat{\beta}(1)=q_{\mathds{1}}^{-1}\circ \beta(1)=q_{\mathds{1}}^{-1}({\bf 0})=\mathds{1}$, i.e., $h(\beta)=[\beta]\in K$. Hence $\alpha\in W\subseteq h^{-1}(K)$.
\end{proof}

Given a subgroup $H$ of a group $G$, recall that the largest normal subgroup $N$ of $G$ contained in $H$ is given by $N=\bigcap_{g\in G} gHg^{-1}$ and is called the {\em core} of $H$ in $G$. If $N=\{1\}$, we call $H$ a {\em core-free} subgroup of $G$. We now show that $K$ is a core-free subgroup of $\pi_1(\mathds{H},{\bf 0})$.

\begin{proposition}\label{normal}
$K$ does not contain any nontrivial normal subgroup of $\pi_1(\mathds{H},{\bf 0})$.
\end{proposition}

\begin{proof}
Let $1\neq[\alpha]\in K$. Consider the maps $f_n:\mathds{H}\rightarrow \bigcup_{i=1}^n C_i$ defined by $f_n(x)=x$ if $x\in C_i$ with $1\leqslant i \leqslant n$ and $f_n(x)={\bf 0}$ otherwise. Put $\alpha_n=f_n\circ \alpha$. By \cite[Theorem~4.1]{MM}, there is an $n\in \mathbb{N}$ such that $1\neq [\alpha_n]\in \pi_1(\bigcup_{i=1}^n C_i, {\bf 0})$.  Choosing a different representative for $[\alpha]$, if necessary, we may assume that there is a partition $0=t_0<t_1<t_2<\cdots <t_m=1$ of $[0,1]$ and a word $w=a_{s_0}^{\epsilon_0} a_{s_1}^{\epsilon_1} \cdots a_{s_{(m-3)/2}}^{\epsilon_{(m-3)/2}}\in {\mathcal W}_n$ with $\epsilon_{k}\in \{+1,-1\}$ and $w'\neq \mathds{1}$, such that $\alpha([t_{2k},t_{2k+1}])\subseteq \bigcup_{i=n+1}^\infty C_i$, and $\alpha(t)=l_{s_k}((t-t_{2k+1})/(t_{2k+2}-t_{2k+1}))$ for all $t\in [t_{2k+1},t_{2k+2}]$ if $\epsilon_{k}=+1$ and $\alpha(t)=l_{s_k}((t-t_{2k+2})/(t_{2k+1}-t_{2k+2}))$ for all $t\in [t_{2k+1},t_{2k+2}]$ if $\epsilon_{k}=-1$. Since $\{w_1, w_2, w_3,\cdots\}$ is a complete list of all non-empty words in $\mathcal W$, there is a $j\in \mathbb{N}$ such that $w=w_j$.
Let $\beta:([0,1],\{0,1\})\rightarrow (\mathds{H},{\bf 0})$ be a path which alternates between $l_1$ and $l_2$ according to the finite word $\widehat{w}_j$, and let $\beta^-(t)=\beta(1-t)$. Consider the lift $\widehat{\gamma}$ of $\gamma=\beta\cdot \alpha \cdot \beta^-$. Then $\widehat{\gamma}(0)=\mathds{1}$ and $\widehat{\gamma}(1/3)=\widehat{w}_j\in \widehat{Y}_j$. By Lemma~\ref{islands}(3), every subpath of $\widehat{\gamma}|_{[1/3,2/3]}$ which lifts one of the $\alpha|_{[t_{2k},t_{2k+1}]}$ is a loop. Hence $\widehat{\gamma}|_{[1/3,2/3]}$ visits the same vertices as the edge-path described in Lemma~\ref{islands}(5). Since $\Gamma^\ast$ is a tree and since $w_j'\neq \mathds{1}$, we have $\widehat{\gamma}(1/3)\neq  \widehat{\gamma}(2/3)$. Consequently, $\beta\cdot \alpha \cdot \beta^-$ does not lift to a loop at $\mathds{1}$. Hence,  $[\beta][\alpha][\beta]^{-1}\not\in K$.
\end{proof}

\begin{remark}
Below  we will see (in Corollary~\ref{nocov}) that there is no covering projection $p:(\widetilde{X},\widetilde{x})\rightarrow (\mathds{H},{\bf 0})$ such that  $p_{\#}(\pi_1(\widetilde{X},\widetilde{x}))=K$.
\end{remark}

\section{The classical (semi)covering construction: \\ open versus open normal subgroups}

  Given a connected and locally path-connected space $X$ and a subgroup $H$ of $\pi_1(X,x)$, we recall the set-up from the proof of
  \cite[Theorem 2.5.13]{Spanier}. On the set of continuous paths $\alpha:([0,1],0)\rightarrow (X,x)$, consider the equivalence relation $\alpha\sim \beta$ iff $\alpha(1)=\beta(1)$ and $[\alpha \cdot\beta^-]\in H$, where $\beta^-(t)=\beta(1-t)$. Denote the equivalence class of $\alpha$ by $\left<\alpha\right>$ and denote the set of all equivalence classes by $\widetilde{X}$. A basis for the topology of $\widetilde{X}$ is given by all elements of the form \[\left<\alpha,U\right>=\{\left<\alpha\cdot\gamma\right>\mid \gamma:([0,1],0)\rightarrow (U,\alpha(1))\},\] where $U$ is an open subset of  $X$ and $\left<\alpha\right>\in \widetilde{X}$ with $\alpha(1)\in U$. (Note that $\left<\beta\right>\in \left<\alpha,U\right>$ implies $\left<\beta,U\right>=\left<\alpha,U\right>$ and that $V\subseteq U$ implies $\left<\alpha,V\right>\subseteq \left<\alpha,U\right>$.)

The space $\widetilde{X}$ is connected and locally path-connected and the map $p:\widetilde{X}\rightarrow X$, given by $p(\left<\alpha\right>)=\alpha(1)$, is a continuous open surjection.
Here are the two issues:\vspace{5pt}

\noindent   {\bf A. Evenly covered neighborhoods.} Any two basis elements of the form $\left<\alpha,U\right>$ and $\left<\beta,U\right>$ are either disjoint or identical. Moreover, if $U$ is a path-connected open neighborhood of some $u\in X$, then \[p^{-1}(U)=\bigcup_{\left<\alpha\right>\in p^{-1}(u)} \left<\alpha,U\right>.\]

\begin{issue}
When are the maps $p|_{\left<\alpha,U\right>}:\left<\alpha,U\right>\rightarrow U$ homeomorphisms?
\end{issue}

\noindent {\bf B. Standard lifts.} Suppose $Y$ is connected and locally path-connected, $f:Y\rightarrow X$ a continuous map, $y\in Y$ and $\left<\alpha\right>\in \widetilde{X}$ with $p(\left<\alpha\right>)=f(y)$. Then there is a continuous lift $\widetilde{f}:(Y,y)\rightarrow (\widetilde{X},\left<\alpha\right>)$ such that $p\circ \widetilde{f}=f$, provided $f_\#(\pi_1(Y,y))\subseteq [\alpha^{-}]H[\alpha]$.
For example, we may define $\widetilde{f}(z)=\left<\alpha\cdot(f\circ\tau)\right>$, where $\tau:[0,1]\rightarrow Y$ is any continuous path from $\tau(0)=y$ to $\tau(1)=z$.
 Note that $[\alpha^{-}]H[\alpha]\subseteq p_\#(\pi_1(\widetilde{X},\left<\alpha\right>))$. Moreover, if $p:\widetilde{X}\rightarrow X$ has unique path lifting, then $p_\#:\pi_1(\widetilde{X},\left<\alpha\right>)\rightarrow \pi_1(X,f(y))$ is a monomorphism onto $[\alpha^-]H[\alpha]$. (See, for example, \cite[Proposition~6.9]{FischerZastrow}.)

  \begin{issue} When are the lifts $\widetilde{f}$ unique?
\end{issue}

 The lifts $\widetilde{f}$ will be unique if $p:\widetilde{X}\rightarrow X$ has {\em unique path lifting} (UPL), which makes it a Serre fibration. Note that for $p:\widetilde{X}\rightarrow X$ to have UPL, it need not have evenly covered neighborhoods or be a local homeomorphism---it might even have some non-discrete fibers. Indeed, as was shown in \cite[Theorem~6.10]{FischerZastrow}, $p:\widetilde{X}\rightarrow X$ has UPL if $H$ is the kernel of the natural homomorphism $\pi(X,x)\rightarrow \check{\pi}_1(X,x)$ to the first \v{C}ech homotopy group. For example, when
$X=\mathds{H}$, this kernel equals $H=\{1\}$. The resulting map  $p:\widetilde{\mathds{H}}\rightarrow \mathds{H}$ has one exceptional (non-discrete) fiber \cite[Example~4.15]{FischerZastrow}. Moreover, for $p:\widetilde{\mathds{H}}\rightarrow \mathds{H}$, the (unique) lifts of paths and their homotopies do not vary continuously in the compact-open topology. (Indeed, similar to the proof of Lemma~\ref{same} below, one can readily construct a sequence $\tau_n:([0,1],0)\rightarrow (\mathds{H},{\bf 0})$ of reparametrizations of the loops $l_1\cdot l_n \cdot l_1$, which converge to $\tau=l_1\cdot l_1$ in the compact-open topology, but whose lifts $\widetilde{\tau}_n:([0,1],0)\rightarrow (\widetilde{\mathds{H}},\ast)$ do not even converge pointwise to the lift $\widetilde{\tau}:([0,1],0)\rightarrow (\widetilde{\mathds{H}},\ast)$ of $\tau$.)

In contrast, a local homeomorphism has discrete fibers. If a local homeomorphism  $p:\widehat{X}\rightarrow X$ has unique lifts of paths and their homotopies, then classical arguments show that it also has the above (unique) standard lifts subject to the standard criterion: $f_\#(\pi_1(Y,y))\subseteq p_\#(\pi_1(\widehat{X},\widehat{x}))$, where $p(\widehat{x})=f(y)$.  Moreover, $p_\#(\pi_1(\widehat{X},\widehat{x}))$ is open in $\pi_1(X,x)$ for every $\widehat{x}\in p^{-1}(x)$. The proof of the latter fact is a slight modification of the proof of Proposition~\ref{open}. (The only adjustment one needs to make is to include sets of the form $S(\{t_i\},p(U_i\cap U_{i+1}))$ into the intersection defining $W$, in case $p(U_i\cap U_{i+1})\neq p(U_i)\cap p(U_{i+1})$.) A straightforward variation of this proof also shows that all lifts of paths and their homotopies vary continuously in the compact-open topology.

\begin{remark}\label{Hausdorff} It might be worth noting that a local homeomorphism with Hausdorff domain is a semicovering if and only if all lifts of paths and their homotopies {\em exist}. This follows from the previous paragraph and (the natural modification of) the proof of Lemma~\ref{UPC} below.
\end{remark}

We call $X$ {\em homotopically Hausdorff relative to $H$} if every fiber of $p:\widetilde{X}\rightarrow X$ is T$_1$. It is shown in \cite[Proposition~6.4]{FischerZastrow} that for $p:\widetilde{X}\rightarrow X$ to have UPL, $X$ must be homotopically Hausdorff relative to $H$. The following lemma has the same proof as \cite[Lemma~2.1]{FischerZastrow}.

\begin{lemma}\label{same}
If $H$ is open in $\pi_1(X,x)$, then all fibers of $p:\widetilde{X}\rightarrow X$ are discrete.
\end{lemma}

\begin{proof}
 Let $h:\Omega(X,x)\rightarrow \pi_1(X,x)$ denote the quotient map. Let $\left<\alpha\right>\in \widetilde{X}$.
Since $1=[\alpha\cdot \alpha^-]\in H$ and $H$ is open in $\pi_1(X,x)$, there are compact subsets $A_i$ of $[0,1]$ and open subsets $V_i$ of $X$ such that $\alpha\cdot \alpha^-\in \bigcap_{i=1}^n S(A_i, V_i)\subseteq h^{-1}(H)$. Choose a path-connected open neighborhood $U$ of $\alpha(1)$ in $X$ such that $U\subseteq V_i$ whenever $1/2\in A_i$. Now, let $\left<\beta\right>\in \left<\alpha,U\right>$ with $p(\left<\beta\right>)=p(\left<\alpha\right>)$. Then $\left<\beta\right>=\left<\alpha\cdot\gamma\right>$ for some loop $\gamma$ in $U$. By choice of $U$, there is a reparametrization $\tau$ of $\alpha\cdot\gamma\cdot\alpha^-$ such that $\tau\in \bigcap_{i=1}^n S(A_i, V_i)\subseteq h^{-1}(H)$. Hence $[\alpha\cdot\gamma\cdot\alpha^-]\in H$, so that $\left<\beta\right>=\left<\alpha\right>$.
\end{proof}

Put $\pi(\alpha, U)=[\alpha]i_\#(\pi_1(U,\alpha(1)))[\alpha^-]\leqslant \pi_1(X,x)$, where $i:U\hookrightarrow X$ is inclusion.

\vspace{5pt}

Comparing the definitions of $\left<\alpha,U\right>$ and $\pi(\alpha,U)$, we observe:

\begin{lemma}\label{fibers}
  Let  $U$ be an open neighborhood of some  $u\in X$ and $\left<\alpha\right>\in p^{-1}(u)$. Then $\left<\alpha,U\right>\cap p^{-1}(u)=\{\left<\alpha\right>\}$ if and only if  $\pi(\alpha,U)\subseteq H$.
\end{lemma}

In particular, if $\left<\alpha\right>$ is an isolated point of a fiber $p^{-1}(u)$ of $p:\widetilde{X}\rightarrow X$, then there is a path-connected open neighborhood $U$ of $u=\alpha(1)$ in $X$ such that $\pi(\alpha,U)\subseteq H$.

\begin{lemma}\label{inH} Let $\left<\alpha\right>\in \widetilde{X}$ and let $U$ be a path-connected open neighborhood of $\alpha(1)$ in $X$. Then $\pi(\alpha,U)\subseteq H$ if and only if $p|_{\left<\alpha,U\right>}:\left<\alpha,U\right>\rightarrow U$ is a homeomorphism.
\end{lemma}

\begin{proof}
Since $U$ is path-connected, $p|_{\left<\alpha,U\right>}:\left<\alpha,U\right>\rightarrow U$ is a continuous open surjection. Hence, it suffices to show that $\pi(\alpha,U)\subseteq H$ if and only if $p|_{\left<\alpha,U\right>}:\left<\alpha,U\right>\rightarrow U$ is injective.
To this end, suppose $\pi(\alpha,U)\subseteq H$. Let $\left<\alpha\cdot \gamma\right>, \left<\alpha\cdot \delta\right>\in  \left<\alpha,U\right>$ with $\gamma, \delta:([0,1],0)\rightarrow (U, \alpha(1))$ and $p(\left<\alpha\cdot \gamma\right>)=p(\left<\alpha\cdot \delta\right>)$. Then $[\alpha\cdot( \gamma\cdot\delta^-)\cdot\alpha^-]\in \pi(\alpha,U)\subseteq H$. Hence $\left<\alpha\cdot\gamma\right>=\left<\alpha\cdot\delta\right>$. The converse is similar.
\end{proof}

\begin{lemma}\label{UPC}
If $p:\widetilde{X}\rightarrow X$ is a local homeomorphism, then it has UPL.
\end{lemma}

\begin{proof}
 Let $g,h:[0,1]\rightarrow \widetilde{X}$ be two continuous paths with $p\circ g=p\circ h$. We show that $E=\{t\in [0,1]\mid g(t)=h(t)\}$ is both closed and open in $[0,1]$. (i) Let $t\in [0,1]\setminus E$. Say, $g(t)=\left<\alpha\right>$ and $h(t)=\left<\beta\right>$. Since the fibers of $p:\widetilde{X}\rightarrow X$ are discrete, they are T$_1$. So, we may choose an open subset $U$ of $X$ with $\alpha(1)\in U$ such that $h(t)\not\in \left<\alpha,U\right>$. Then
$\left<\alpha,U\right>\cap \left<\beta,U\right>= \emptyset$. Choose an open subset $V$ of $[0,1]$ with $t\in V$ such that $g(V)\subseteq \left<\alpha,U\right>$ and $h(V)\subseteq \left<\beta,U\right>$. Then $t\in V\subseteq [0,1]\setminus E$.
(ii) Let $t\in E$. Choose an open neighborhood $\widetilde{U}$ of $g(t)=h(t)$ such that $U=p(\widetilde{U})$ is open in $X$ and $p|_{\widetilde{U}}:\widetilde{U}\rightarrow U$ is a homeomorphism. Then $t\in g^{-1}(\widetilde{U})\cap h^{-1}(\widetilde{U})\subseteq E$.
\end{proof}

Applying the usual lifting classification \cite[2.5.2]{Spanier} to the above, we obtain:

\begin{corollary} \cite{Brazas2012}\label{B}
The connected semicovering spaces of a connected and locally path-connected topological space $X$ are classified by the conjugacy classes of the open subgroups of $\pi_1(X,x)$.
\end{corollary}

\begin{remark}
The classification of semicoverings given in \cite[Corollary 7.20]{Brazas2012} holds for more general spaces, namely for so-called locally wep-connected spaces. Also note that if  $H_1\subseteq H_2$   are two subgroups of $\pi_1(X,x)$ such that $H_1$ is open in  $\pi_1(X,x)$, then $H_2$ is open in $\pi_1(X,x)$, because it equals a union of cosets of $H_1$.
\end{remark}

\begin{lemma}\label{wind}
 Let $\left<\alpha\right>,\left<\beta\right>\in \widetilde{X}$ with $p(\left<\alpha\right>)=p(\left<\beta\right>)\in U$ for some path-connected open subset $U$ of $X$ for which $p|_{\left<\alpha,U\right>}:\left<\alpha,U\right>\rightarrow U$ is a homeomorphism.

 Then $\pi(\alpha,U)\subseteq H$. Moreover, if $g \pi(\alpha,U)g^{-1}\subseteq H$ with $g=[\beta\cdot \alpha^-]\in \pi_1(X,x)$, then $p|_{\left<\beta,U\right>}:\left<\beta,U\right>\rightarrow U$ is also a homeomorphism.
\end{lemma}

\begin{proof}
This follows from Lemma~\ref{inH}, since $\pi(\beta,U)=g\pi(\alpha,U) g^{-1}$.
\end{proof}

 It follows that, if $H$ is open and normal in $\pi_1(X,x)$, then $p:\widetilde{X}\rightarrow X$ is a covering projection.
  Conversely, suppose $p:(\widetilde{X},\widetilde{x})\rightarrow (X,x)$ is a covering projection with $p_\#(\pi_1(\widetilde{X},\widetilde{x}))=H$. We then return to the line of argument used in \cite[2.5.11 and 2.5.13]{Spanier}. If $U$ is a path-connected open subset of~$X$ that is evenly covered by $p:\widetilde{X}\rightarrow X$, then $\pi(\alpha,U)\subseteq p_\#(\pi_1(\widetilde{X},\widetilde{x}))=H$ for all $\alpha$.
  The subgroup $N$ of $\pi_1(X,x)$ generated by all such $\pi(\alpha,U)$ is a normal subgroup of $\pi_1(X,x)$ which is contained in $H$. If we apply the above construction with $H$ replaced by any subgroup of $\pi_1(X,x)$ containing $N$, we obtain a covering projection by Lemma~\ref{inH}. So, we also  obtain:

\begin{corollary}\cite{Torabi}\label{T}
The connected covering spaces of a connected and locally path-connected topological space $X$ are classified by the conjugacy classes of those (open) subgroups of $\pi_1(X,x)$ which contain an open normal subgroup of $\pi_1(X,x)$.
\end{corollary}

\begin{corollary}\label{nocov}
There is no covering projection $p:(\widetilde{X},\widetilde{x})\rightarrow (\mathds{H},{\bf 0})$ such that  $p_{\#}(\pi_1(\widetilde{X},\widetilde{x}))=K$.
\end{corollary}

\begin{proof}
Suppose, to the contrary, that there is a covering projection $p:(\widetilde{X},\widetilde{x})\rightarrow (\mathds{H},{\bf 0})$ with $p_{\#}(\pi_1(\widetilde{X},\widetilde{x}))=K$. Then, by Corollary~\ref{T}, there exists an open normal subgroup $N$ of $\pi_1(\mathds{H},{\bf 0})$ with $N\subseteq K$. By Proposition~\ref{normal}, we must have $N=\{1\}$. This implies that $\pi_1(\mathds{H},{\bf 0})$ is discrete.  The latter can only hold if $\mathds{H}$ is semilocally simply-connected \cite[Lemma~3.1]{Calcut}, but it is not.
\end{proof}

\noindent {\bf Acknowledgements}.  This work was partially supported by a grant from the Simons Foundation (\#245042 to Hanspeter Fischer). The authors would also like to thank the referee for comments that helped improve the exposition of the paper.


\begin{thebibliography}{00}

\bibitem{Biss} D.K. Biss, {\em The topological fundamental group and generalized covering spaces}, Topology and its Applications {\bf 124 } (2002) 355--371.

\bibitem{BrazasNote} J. Brazas, {\em Regular coverings, Spanier groups, and a
topologized fundamental group},\linebreak \url{http://www2.gsu.edu/~jbrazas/Regularcoverings.pdf} (26 November 2011).

\bibitem{Brazas2011} J. Brazas, {\em The topological fundamental group and free topological groups}, Topology and its Applications {\bf 158} (2011) 779--802.

\bibitem{Brazas2012} J. Brazas, {\em Semicoverings: a generalization of covering space theory}, Homology, Homotopy and Applications {\bf 14} (2012) 33--63.

\bibitem{Brazas2013} J. Brazas, {\em The fundamental group as a topological group}, Topology and its Applications {\bf 160} (2013) 170--188.

\bibitem{Calcut} J.S. Calcut and J.D. McCarthy,  {\em Discreteness and homogeneity of the topological fundamental group}, Topology Proceedings {\bf 34} (2009) 339--349.

\bibitem{Fabel}    P. Fabel, {\em  Multiplication is discontinuous in the Hawaiian earring group (with the quotient topology)}, Bulletin of the Polish Academy of Sciences. Mathematics {\bf 59 } (2011) 77--83.


    \bibitem{FischerZastrow} H. Fischer and A. Zastrow, {\em Generalized universal covering spaces and the shape group}, Fundamenta Mathematicae {\bf 197} (2007) 167--196.

\bibitem{Hatcher} A. Hatcher, Algebraic Topology, Cambridge University Press, 2001.

\bibitem{MM} J.W. Morgan and I. Morrison, {\em A van Kampen theorem for weak joins},  Proceedings of the London Mathematical Society {\bf 53} (1986) 562--576.

\bibitem{Spanier} E.H. Spanier, Algebraic Topology, McGraw-Hill, 1966.

\bibitem{Torabi} H. Torabi, A. Pakdaman and B. Mashayekhy, {\em On the Spanier groups and covering and semicovering spaces}, Preprint (18 July 2012) arXiv:1207.4394v1.

\end{thebibliography}
\end{document}